\documentclass[12pt]{amsart}

\usepackage[
    a4paper,
    left=2cm,
    right=2cm,
    top=3cm,
    bottom=3cm
]{geometry}

\usepackage{amsmath,amssymb,amsfonts,amsthm,mathtools}
\usepackage{microtype}
\usepackage{enumitem}
\usepackage{hyperref}
\usepackage{mathptmx}
\usepackage{tikz}

\usetikzlibrary{positioning}

\allowdisplaybreaks

\setlist{
    leftmargin=*,
    topsep=2pt,
    itemsep=0pt,
    parsep=0pt,
    partopsep=0pt
}

\hypersetup{
    colorlinks=true,
    linkcolor=black,
    citecolor=black,
    urlcolor=black
}

\tikzset{
    vertex/.style={
        circle,
        fill=black,
        draw=black,
        inner sep=2.5pt
    },
    edge/.style={
        draw=black,
        line width=1.15pt
    },
    graphlabel/.style={
        font=\small
    }
}

\makeatletter

\def\@settitle{%
    \begin{center}
        \normalfont\Large\bfseries
        \@title
    \end{center}
}

\renewcommand{\section}{%
    \@startsection
        {section}
        {1}
        {\z@}
        {-2.1ex \@plus -0.6ex \@minus -0.2ex}
        {0.8ex \@plus 0.2ex}
        {\normalfont\bfseries\scshape\centering}%
}

\renewcommand{\subsection}{%
    \@startsection
        {subsection}
        {2}
        {\z@}
        {-1.6ex \@plus -0.4ex \@minus -0.2ex}
        {0.55ex \@plus 0.15ex}
        {\normalfont\normalsize\bfseries}%
}

\makeatother

\newtheoremstyle{compactplain}
{3pt}
{3pt}
{\itshape}
{}
{\bfseries}
{.}
{0.5em}
{}

\newtheoremstyle{compactdefinition}
{3pt}
{3pt}
{\normalfont}
{}
{\bfseries}
{.}
{0.5em}
{}

\newtheoremstyle{compactremark}
{3pt}
{3pt}
{\normalfont}
{}
{\itshape}
{.}
{0.5em}
{}

\newtheorem{theorem}{Theorem}[section]
\newtheorem{proposition}[theorem]{Proposition}
\newtheorem{lemma}[theorem]{Lemma}
\newtheorem{corollary}[theorem]{Corollary}

\newtheorem{definition}[theorem]{Definition}
\newtheorem{example}[theorem]{Example}

\newcommand{\Ann}{\operatorname{Ann}}
\newcommand{\Spec}{\operatorname{Spec}}
\newcommand{\diam}{\operatorname{diam}}

\title[Multiplicity-One Structure in Graded Zero-Divisor Graphs]
{\normalfont\textbf{Multiplicity-One Structure in Graded Zero-Divisor Graphs}}

\author{\textsc{M. Aijaz}\textsuperscript{1}, \textsc{Shahnawaz Ahmad Rather}\textsuperscript{2}, \textsc{S. Pirzada}\textsuperscript{3}}

\thanks{\textsuperscript{1}Government Degree College for Women Baramulla, Jammu and Kashmir, India. E-mail: \href{mailto:ahaijaz99@gmail.com}{ahaijaz99@gmail.com}.}

\thanks{\textsuperscript{2}Government Degree College Baramulla, Jammu and Kashmir, India. E-mail: \href{mailto:nawazcuj@gmail.com}{nawazcuj@gmail.com}.}

\thanks{\textsuperscript{3}University of Kashmir, Srinagar, Jammu and Kashmir, India. E-mail: \href{mailto:pirzadasd@kashmiruniversity.ac.in}{pirzadasd@kashmiruniversity.ac.in}.}

\date{}

\begin{document}

\begin{abstract}
Let \(H\) be a multiplicative group, \(R\) a commutative \(H\)-graded ring, and \(M\) a nonzero finitely generated gr-semisimple \(H\)-graded \(R\)-module. We assume that \(M\) admits a fixed decomposition into powers of gr-simple modules with pairwise distinct annihilators, each repeated block carrying the componentwise grading. For a nonzero homogeneous element \(x\), we compute the colon ideal \((Rx:M)\) and show that it is determined by the multiplicity-one blocks on which \(x\) has a nonzero component. This leads to the criterion
\(
x\leftrightarrow y
\text{ if and only if }
\sigma(x)\cap\sigma(y)=\varnothing,
\)
where \(x\leftrightarrow y\) means that \(x\) and \(y\) are adjacent in \(\Gamma_H(M)\). Thus adjacency is governed by disjointness of multiplicity-one obstruction supports. The corresponding obstruction classes form a weighted disjointness structure: higher multiplicities affect the sizes of the classes, but not the adjacency relation between their supports. We use this description to study graph invariants, reconstruction, and Kneser-type layers, and, in the finite case, to obtain formulas for the obstruction-class sizes and a reduction of the adjacency spectrum.
\end{abstract}

\maketitle

\noindent \textbf{Keywords:} graded module, zero-divisor graph, gr-semisimple module, semisimple decomposition, annihilator, clique number.

\smallskip

\noindent \textbf{2020 Mathematics Subject Classification:} 05C25, 05C15, 05C50, 13A02, 16W50.

\section{Introduction}

Zero-divisor graphs provide a way of relating algebraic properties to the structure of a graph. The subject has its roots in the work of Beck \cite{Beck1988}, while the standard zero-divisor graph was developed by Anderson and Livingston \cite{AndersonLivingston1999}. Related constructions for modules were studied in \cite{Behboodi2012,Safaeeyan2014}, and chromatic properties of module zero-divisor graphs were considered in \cite{LeeVarmazyar2019}.

The graded zero-divisor graph \(\Gamma_H(M)\) was introduced in \cite{MohdEtAl2026}, where its connectivity, girth, and related properties were investigated. In this paper we consider the adjacency relation itself for a class of graded semisimple modules. The main algebraic ingredient is the colon ideal \((Rx:M)\) associated with a nonzero homogeneous element \(x\). An explicit description of this ideal leads to a simple criterion for adjacency.

Throughout the paper, the prefix \emph{gr-} refers to the graded category. Thus, a nonzero graded module is \emph{gr-simple} if it has no nonzero proper graded submodule, and it is \emph{gr-semisimple} if it is a direct sum of gr-simple graded submodules; see \cite{Dade1980,Nastasescu1982}. We assume that \(M\neq0\) is a finitely generated gr-semisimple \(H\)-graded \(R\)-module with a fixed decomposition \(\displaystyle M=\bigoplus_{\alpha=1}^{t}E_\alpha,~ E_\alpha=S_\alpha^{m_\alpha}, \) where each \(S_\alpha\) is gr-simple and the ideals \(\Ann_R(S_1),\ldots,\Ann_R(S_t)\) are pairwise distinct. Each block \(E_\alpha\) carries the componentwise grading, so its copies of \(S_\alpha\) have the same grading. Copies with different grading shifts are not grouped into the same repeated block. The term \emph{isotypic block} will always refer to this fixed decomposition. Since the annihilator of a gr-simple module is gr-maximal, the distinct-annihilator assumption gives the pairwise comaximality needed later to separate the blocks.

Let \({I_1}=\{\alpha:m_\alpha=1\}\), \({I_{\ge 2}}=\{\alpha:m_\alpha\ge2\}\), and \(s=|{I_1}|\). For a nonzero homogeneous element \(\displaystyle x=\sum_{\alpha=1}^{t}x^{(\alpha)}\), define its multiplicity-one obstruction support by \(\sigma(x)=\{\alpha\in{I_1}:x^{(\alpha)}\neq0\}\). Thus \(\sigma(x)\) records exactly those multiplicity-one blocks in which \(x\) has a nonzero component.

We call the fixed decomposition \emph{multiplicity-free} when \(m_\alpha=1\) for every \(\alpha\).

The main algebraic result is
\[
(Rx:M)=\bigcap_{\alpha\notin\sigma(x)}\Ann_R(S_\alpha).
\]
It follows that, for distinct vertices \(x\) and \(y\),
\[
x\leftrightarrow y
\text{ if and only if }
\sigma(x)\cap\sigma(y)=\varnothing.
\]
Accordingly, adjacency depends only on the multiplicity-one obstruction supports. Blocks of multiplicity at least two do not enter the adjacency criterion, although they may change the sizes of the corresponding obstruction classes.

This gives \(\Gamma_H(M)\) the structure of a weighted disjointness graph. We use it to determine clique and chromatic numbers, distances and girth, and to study reconstruction from the abstract graph. The same description yields Kneser-type subgraphs. For finite modules it also gives formulas for obstruction-class sizes and a reduction of the adjacency spectrum to the quotient matrix associated with the obstruction-class partition.

\section{Preliminaries}

Let \(H\) be a multiplicative group with identity \(e\), and let \(R\) be a commutative ring with identity \(1\neq0\). The ring \(R\) is \emph{\(H\)-graded} if \(\displaystyle R=\bigoplus_{g\in H}R_g\), where each \(R_g\) is an additive subgroup of \(R\) and \(R_gR_h\subseteq R_{gh}\) for all \(g,h\in H\). An \(R\)-module \(M\) is \emph{unitary} if \(1m=m\) for every \(m\in M\). A unitary \(R\)-module \(M\) is an \emph{\(H\)-graded \(R\)-module} if \(\displaystyle M=\bigoplus_{g\in H}M_g\), where each \(M_g\) is an additive subgroup of \(M\) and \(R_gM_h\subseteq M_{gh}\) for all \(g,h\in H\). The subgroups \(R_g\) and \(M_g\) are called the homogeneous components of degree \(g\).

An element \(x\in R\) or \(x\in M\) is \emph{homogeneous} if it belongs to some homogeneous component; if \(x\neq0\), that component is unique and its index is called the degree of \(x\). We write \(\displaystyle h(R)=\bigcup_{g\in H}R_g\) and \(\displaystyle h(M)=\bigcup_{g\in H}M_g\) for the sets of homogeneous elements of \(R\) and \(M\), respectively. A submodule \(N\le M\) is \emph{graded} if \(\displaystyle N=\bigoplus_{g\in H}(N\cap M_g)\), and an ideal \(I\) of \(R\) is graded if \(\displaystyle I=\bigoplus_{g\in H}(I\cap R_g)\). If \(N\le M\) is graded, then \(M/N\) carries the quotient grading \((M/N)_g=(M_g+N)/N\). For graded modules \(N_1,\ldots,N_m\), the componentwise grading on \(N_1\oplus\cdots\oplus N_m\) is given by \((N_1\oplus\cdots\oplus N_m)_g=(N_1)_g\oplus\cdots\oplus(N_m)_g\). The \emph{degree support} of a graded module \(N\) is \(\{g\in H:N_g\neq0\}\).

For \(a\in H\), the \emph{grading shift} \(N(a)\) of a graded module \(N\) is the graded module with the same underlying \(R\)-module and homogeneous components \(N(a)_g=N_{ga}\) for \(g\in H\).

As in \cite{MohdEtAl2026}, the prefix \emph{gr-} indicates that the corresponding notion is taken in the graded category. A cyclic submodule \(Rx\) is \emph{gr-cyclic} when \(x\) is homogeneous. A nonzero graded module is \emph{gr-simple} if it has no nonzero proper graded submodule, and it is \emph{gr-semisimple} if it is a direct sum of gr-simple graded submodules. A proper graded submodule is \emph{gr-maximal} if it is maximal among proper graded submodules; in particular, a proper graded ideal is gr-maximal when it is maximal among proper graded ideals. A homomorphism \(\varphi:N\to L\) of graded modules is degree-preserving if \(\varphi(N_g)\subseteq L_g\) for every \(g\in H\); a bijective degree-preserving homomorphism is a graded isomorphism. Standard terminology from graded ring and module theory is used throughout; see \cite{Dade1980,Nastasescu1982}.

For a submodule \(N\le M\), set \((N:_RM)=\{r\in R:rM\subseteq N\}\), and write simply \((N:M)\) when no ambiguity can arise. For \(X\subseteq M\), let \(\Ann_R(X)=\{r\in R:rX=0\}\). For \(x\in M\), \(Rx=\{rx:r\in R\}\) denotes the cyclic submodule generated by \(x\).

Following \cite{MohdEtAl2026}, for nonzero homogeneous \(x,y\in M\), write \(x\star_M y=0\) whenever \((Ry:M)x=0\) or \((Rx:M)y=0\). Set \(Z_H(M)=\{0\}\cup\{x\in h(M)\setminus\{0\}:x\star_M y=0\text{ for some }0\neq y\in h(M),\ y\neq x\}\). The \emph{graded zero-divisor graph} \(\Gamma_H(M)\) is the simple graph with vertex set \(Z_H(M)^*=Z_H(M)\setminus\{0\}\), in which distinct vertices \(x\) and \(y\) are adjacent precisely when \(x\star_M y=0\). We write \(x\leftrightarrow y\) for adjacency in \(\Gamma_H(M)\). Under the trivial grading, where \(R_e=R\), \(M_e=M\), and all other homogeneous components are zero, this reduces to the corresponding ungraded zero-divisor graph of a module.

All graphs are simple. For a graph \(G\), we write \(V(G)\) and \(E(G)\) for its vertex and edge sets, \(N_G(v)\) and \(N_G[v]\) for the open and closed neighborhoods of \(v\), and \(d_G(u,v)\) for the distance between \(u\) and \(v\). The symbols \(\omega(G)\), \(\chi(G)\), \(\operatorname{diam}G\), and \(g(G)\) denote the clique number, chromatic number, diameter, and girth, respectively; we take \(g(G)=\infty\) when \(G\) is acyclic. A graph is \emph{perfect} if every finite induced subgraph \(F\) satisfies \(\chi(F)=\omega(F)\). For finite \(G\), \(\Spec(G)\) denotes the adjacency spectrum with multiplicities, and the same notation is used for the eigenvalue multiset of a square matrix. We use the standard notation \(K_n\), \(K_{a,b}\), \(\overline K_n\), and \(G\vee H\); see \cite{BondyMurty2008}.

A vertex of a graph is \emph{universal} if it is adjacent to every other vertex; a clique consisting entirely of universal vertices will be called a universal clique.

An \emph{abstract simplicial complex} on a set \(X\) is a family \(\Delta\) of subsets of \(X\) such that \(\varnothing\in\Delta\) and \(Q\subseteq P\in\Delta\) implies \(Q\in\Delta\).

We shall use the following general properties of the graded zero-divisor graph established in \cite{MohdEtAl2026}.

\begin{lemma}
\label{lem:prelim-Mohd}
Let \(x,y\in h(M)\setminus\{0\}\).
\begin{enumerate}
\item If \(x\leftrightarrow y\), then \(rx\star_M sy=0\) for all \(r,s\in h(R)\) such that \(rx\neq0\) and \(sy\neq0\). In particular, if \(rx\neq sy\), then \(rx\leftrightarrow sy\).
\item If \(Rx\cap Ry=0\), then \(x\) and \(y\) are vertices of \(\Gamma_H(M)\), and \(x\leftrightarrow y\).
\end{enumerate}
\end{lemma}

\begin{theorem}
\label{thm:prelim-diameter}
If \(V(\Gamma_H(M))\neq\varnothing\), then \(\Gamma_H(M)\) is connected and \(\operatorname{diam}\Gamma_H(M)\le3\).
\end{theorem}

\begin{theorem}
\label{thm:prelim-girth}
The graph \(\Gamma_H(M)\) is acyclic or has girth at most \(4\); equivalently, \(g(\Gamma_H(M))=\infty\) or \(g(\Gamma_H(M))\le4\).
\end{theorem}

We shall also use the standard fact that if \(S\) is gr-simple, then \(\Ann_R(S)\) is a gr-maximal ideal of \(R\); see \cite{Nastasescu1982}.

\section{Algebraic Structure of Cyclic Submodules}

Assume \(\displaystyle M=\bigoplus_{\alpha=1}^{t}E_\alpha\), where
\(E_\alpha=S_\alpha^{m_\alpha}\). Put
\(\displaystyle \mathfrak m_\alpha=\Ann_R(S_\alpha)\).

Since \(S_\alpha\) is gr-simple, \(\mathfrak m_\alpha\) is a gr-maximal ideal of \(R\); see \cite{Nastasescu1982}. By the standing hypothesis, the ideals
\(\mathfrak m_1,\ldots,\mathfrak m_t\) are distinct, and hence pairwise comaximal. Therefore, by the Chinese remainder theorem, for every \(\alpha\) there exists \(e_\alpha\in R\) such that \(e_\alpha\equiv1\pmod{\mathfrak m_\alpha}\) and \(e_\alpha\equiv0\pmod{\mathfrak m_\beta}\) for every \(\beta\neq\alpha\). The elements \(e_\alpha\) need not be homogeneous.

For \(\displaystyle x=\sum_{\alpha=1}^{t}x^{(\alpha)}\), write \(A(x)=\{\alpha:x^{(\alpha)}\neq0\}\). The selectors show that no isotypic coordinate occurring in \(x\) is lost upon passing to \(Rx\).

\begin{lemma}
\label{lemma:3.1}
For every \(\alpha\in A(x)\), we have \(x^{(\alpha)}\in Rx\). Consequently, \(\displaystyle Rx=\bigoplus_{\alpha\in A(x)}Rx^{(\alpha)}\).
\end{lemma}

\begin{proof}
Since \(E_\gamma=S_\gamma^{m_\gamma}\), we have \(\Ann_R(E_\gamma)=\mathfrak m_\gamma\). Fix \(\alpha\in A(x)\) and choose the selector \(e_\alpha\). Then, \(e_\alpha-1\in\mathfrak m_\alpha\) and \(e_\alpha\in\mathfrak m_\beta\) for \(\beta\neq\alpha\), so \( e_\alpha x=x^{(\alpha)}\). Hence \(x^{(\alpha)}\in Rx\). It follows that \(\displaystyle Rx=\sum_{\alpha\in A(x)}Rx^{(\alpha)}\), and the sum is direct because \(Rx^{(\alpha)}\subseteq E_\alpha\) and \(M=\bigoplus_{\alpha=1}^tE_\alpha\).
\end{proof}

\begin{lemma}
\label{lemma:3.2}
If \(0\neq u\in h(E_\alpha)\), then \(Ru\) is gr-simple and \(Ru\cong S_\alpha\) as a graded module.
\end{lemma}

\begin{proof}
Write \(u=(u_1,\ldots,u_{m_\alpha})\in E_\alpha=S_\alpha^{m_\alpha}\). Since the direct sum carries the componentwise grading and \(u\) is homogeneous, every nonzero coordinate \(u_i\) is homogeneous of the same degree as \(u\). Choose \(j\) such that \(u_j\neq0\).

The cyclic submodule \(Ru_j\) is a nonzero graded submodule of the gr-simple module \(S_\alpha\); hence \(Ru_j=S_\alpha\). In particular, \(\Ann_R(u_j)=\Ann_R(S_\alpha)=\mathfrak m_\alpha\).

Let \(\pi_j:E_\alpha\to S_\alpha\) denote the \(j\)-th coordinate projection. Since \(E_\alpha\) carries the componentwise grading, \(\pi_j\) is a degree-preserving graded homomorphism. Its restriction \(\varphi:=\pi_j|_{Ru}:Ru\to S_\alpha\) is surjective, since \(\varphi(Ru)=Ru_j=S_\alpha\).

To prove injectivity, let \(ru\in\ker\varphi\). Then \(ru_j=0\), so \(r\in\Ann_R(u_j)=\mathfrak m_\alpha\). Since \(\mathfrak m_\alpha=\Ann_R(S_\alpha)\), the element \(r\) annihilates every coordinate \(u_i\). Hence \(ru=0\). Thus \(\ker\varphi=0\).

Therefore \(\varphi\) is a graded isomorphism \(Ru\cong S_\alpha\). Since \(S_\alpha\) is gr-simple, so is \(Ru\).
\end{proof}

\begin{proposition}
\label{prop:3.3}
For \(0\neq x\in h(M)\), \(E_\alpha\subseteq Rx \text{ if and only if } m_\alpha=1 \ \text{ and }\ x^{(\alpha)}\neq0\).
\end{proposition}

\begin{proof}
Assume first that \(E_\alpha\subseteq Rx\). By Lemma~\ref{lemma:3.1}, \(\displaystyle Rx=\bigoplus_{\beta\in A(x)}Rx^{(\beta)}, ~Rx^{(\beta)}\subseteq E_\beta\). Since \(E_\alpha\neq0\) and \(E_\alpha\subseteq Rx\), necessarily \(\alpha\in A(x)\); hence \(x^{(\alpha)}\neq0\). Moreover, the directness of the isotypic decomposition gives \(Rx\cap E_\alpha=Rx^{(\alpha)}\). Hence \(E_\alpha=Rx^{(\alpha)}\).

By Lemma~\ref{lemma:3.2}, \(Rx^{(\alpha)}\) is gr-simple. Hence \(E_\alpha\) is gr-simple. Since \(E_\alpha=S_\alpha^{m_\alpha}\), this is impossible when \(m_\alpha\ge2\), because one copy of \(S_\alpha\) is then a nonzero proper graded submodule of \(E_\alpha\). Therefore \(m_\alpha=1\).

Conversely, suppose that \(m_\alpha=1\) and \(x^{(\alpha)}\neq0\). Then \(E_\alpha=S_\alpha\). Since \(x^{(\alpha)}\) is homogeneous and nonzero, gr-simplicity of \(S_\alpha\) gives \(Rx^{(\alpha)}=S_\alpha=E_\alpha\). Lemma~\ref{lemma:3.1} gives \(Rx^{(\alpha)}\subseteq Rx\), and hence \(E_\alpha\subseteq Rx\).
\end{proof}



Proposition~\ref{prop:3.3} identifies the isotypic blocks which are already contained in \(Rx\). It is convenient to record exactly their indices. For \(0\neq x\in h(M)\), put \(\sigma(x)=A(x)\cap{I_1}\). Thus \(\alpha\in\sigma(x)\) precisely when \(E_\alpha\subseteq Rx\). As usual, the intersection of an empty family of ideals is taken to be \(R\).

\begin{lemma}
\label{lemma:3.4}
Let \(E=S^m\), where \(S\) is gr-simple, and let \(r\in h(R)\). If \(L\subsetneq E\) is a graded submodule and \(rE\subseteq L\), then \(rE=0\).
\end{lemma}

\begin{proof}
Suppose that \(rE\neq0\). Then \(rS\neq0\). Since \(r\) is homogeneous, \(rS\) is a nonzero graded submodule of \(S\), and gr-simplicity gives \(rS=S\). Therefore, \(\displaystyle rE=r(S^m)=(rS)^m=S^m=E\), contrary to \(rE\subseteq L\subsetneq E\). Hence \(rE=0\).
\end{proof}

The preceding description of the cyclic submodule \(Rx\) gives an explicit form of its colon ideal.

\begin{theorem}
\label{thm:colon}
For every \(0\neq x\in h(M)\), \(\displaystyle (Rx:M)= \bigcap_{\alpha\notin\sigma(x)} \mathfrak m_\alpha\).
\end{theorem}

\begin{proof}
For each \(\alpha\), let \(L_\alpha=Rx\cap E_\alpha\). By Lemma~\ref{lemma:3.1}, the cyclic submodule \(Rx\) splits along the isotypic decomposition, so \(Rx=\bigoplus_{\alpha=1}^{t}L_\alpha\). Hence, \(M/Rx\cong\bigoplus_{\alpha=1}^{t}E_\alpha/L_\alpha\). Using the standard identity \((Rx:M)=\Ann_R(M/Rx)\), we get \(\displaystyle (Rx:M)=\bigcap_{\alpha=1}^{t}\Ann_R(E_\alpha/L_\alpha)\). If \(\alpha\in\sigma(x)\), Proposition~\ref{prop:3.3} gives \(E_\alpha\subseteq Rx\). Thus \(L_\alpha=E_\alpha\), and so \(\Ann_R(E_\alpha/L_\alpha)=R\).

Now let \(\alpha\notin\sigma(x)\). Then \(L_\alpha\subsetneq E_\alpha\). Since \(\mathfrak m_\alpha=\Ann_R(E_\alpha)\), we have \(\mathfrak m_\alpha\subseteq\Ann_R(E_\alpha/L_\alpha)\). For the reverse inclusion, take a homogeneous element \(r\in\Ann_R(E_\alpha/L_\alpha)\). Then \(rE_\alpha\subseteq L_\alpha\). Since \(L_\alpha\) is a proper graded submodule of \(E_\alpha\), Lemma~\ref{lemma:3.4} gives \(rE_\alpha=0\). Hence \(r\in\mathfrak m_\alpha\). Both ideals are graded, and therefore \(\Ann_R(E_\alpha/L_\alpha)=\mathfrak m_\alpha\).

The factors indexed by \(\sigma(x)\) contribute only the unit ideal, while the remaining factors contribute \(\mathfrak m_\alpha\). Therefore, \(\displaystyle (Rx:M)=\bigcap_{\alpha\notin\sigma(x)}\mathfrak m_\alpha\).
\end{proof}

The colon ideal sees only the multiplicity-one blocks which are completely contained in \(Rx\). The annihilator of \(x\), in contrast, records every isotypic component in which \(x\) has a nonzero coordinate.

\begin{corollary}
\label{cor:3.6}
For \(0\neq x\in h(M)\), \(\displaystyle \Ann_R(x) = \bigcap_{\alpha\in A(x)} \mathfrak m_\alpha\).
\end{corollary}

\begin{proof}
Since the decomposition of \(M\) is direct, \(\displaystyle \Ann_R(x) = \bigcap_{\alpha\in A(x)} \Ann_R(x^{(\alpha)})\). By Lemma~\ref{lemma:3.2},
\(\displaystyle Rx^{(\alpha)}\cong S_\alpha\), and hence \(\Ann_R(x^{(\alpha)})=\mathfrak m_\alpha\). The result follows.
\end{proof}

Pairwise comaximality also makes inclusions between intersections of the ideals \(\mathfrak m_\alpha\) entirely combinatorial.

\begin{lemma}
\label{lem:intersection}
For \(I,J\subseteq\{1,\ldots,t\}\), \(\displaystyle \bigcap_{\alpha\in I}\mathfrak m_\alpha \subseteq \bigcap_{\beta\in J}\mathfrak m_\beta \text{ if and only if } J\subseteq I\).
\end{lemma}

\begin{proof}
Suppose first that \(J\subseteq I\). Then every ideal appearing in the intersection indexed by \(J\) also appears in the intersection indexed by \(I\), so the required inclusion is immediate.

Conversely, assume that \(\bigcap_{\alpha\in I}\mathfrak m_\alpha \subseteq\bigcap_{\beta\in J}\mathfrak m_\beta\). Suppose, toward a contradiction, that \(J\nsubseteq I\). Choose \(\beta\in J\setminus I\). By the choice of the selectors above, there exists \(e_\beta\in R\) such that \(\displaystyle
e_\beta\equiv1\pmod{\mathfrak m_\beta} \text{ and } e_\beta\equiv0\pmod{\mathfrak m_\alpha} (\alpha\neq\beta)\). Since \(\beta\notin I\), we have \(e_\beta\in\bigcap_{\alpha\in I}\mathfrak m_\alpha\). By the assumed inclusion, \(e_\beta\in\bigcap_{\gamma\in J}\mathfrak m_\gamma\). In particular, since \(\beta\in J\), we have \(e_\beta\in\mathfrak m_\beta\).

On the other hand, \(e_\beta\equiv1\pmod{\mathfrak m_\beta}\), so \(e_\beta-1\in\mathfrak m_\beta\). Hence \(1\in\mathfrak m_\beta\), contradicting the fact that \(\mathfrak m_\beta\) is a proper ideal. Therefore \(J\subseteq I\).
\end{proof}
\section{Multiplicity-One Adjacency and Obstruction Supports}

The definition of \(x\star_M y\) is asymmetric in appearance: either \((Rx:M)y=0\) or \((Ry:M)x=0\). The colon formula shows that, in the present semisimple setting, both conditions are equivalent.

\begin{theorem}
\label{thm:symmetric}
For distinct nonzero homogeneous \(x,y\), the following are equivalent:
\begin{enumerate}
\item \((Rx:M)y=0\);
\item \((Ry:M)x=0\);
\item \(\sigma(x)\cap\sigma(y)=\varnothing\).
\end{enumerate}
\end{theorem}

\begin{proof}
By Theorem~\ref{thm:colon}, \(\displaystyle (Rx:M)=\bigcap_{\alpha\notin\sigma(x)}\mathfrak m_\alpha\), and by Corollary~\ref{cor:3.6}, \(\displaystyle \Ann_R(y)=\bigcap_{\beta\in A(y)}\mathfrak m_\beta\).

Now \((Rx:M)y=0\) if and only if \((Rx:M)\subseteq\Ann_R(y)\). By Lemma~\ref{lem:intersection}, this is equivalent to \(A(y)\subseteq\{1,\ldots,t\}\setminus\sigma(x)\). Equivalently, \(A(y)\cap\sigma(x)=\varnothing\). Since \(\sigma(x)\subseteq{I_1}\) and \(\sigma(y)=A(y)\cap{I_1}\), this is the same as \(\sigma(x)\cap\sigma(y)=\varnothing\). Thus \((Rx:M)y=0\) if and only if \(\sigma(x)\cap\sigma(y)=\varnothing\). Interchanging \(x\) and \(y\) gives \((Ry:M)x=0\) under the same conditions.
\end{proof}

We can now state the main adjacency criterion.

\begin{theorem}
\label{thrm:Main-AT}
For distinct  homogeneous elements \(x,y\in M\setminus\{0\}\), \( x\star_M y=0 \text{ if and only if } \sigma(x)\cap\sigma(y)=\varnothing\).
\end{theorem}

\begin{proof}
This follows immediately from Theorem~\ref{thm:symmetric} and the definition of \(\star_M\).
\end{proof}

\begin{corollary}
\label{cor:Main-AT}
If \(x\) and \(y\) are distinct nonzero homogeneous elements and \(\sigma(x)\cap\sigma(y)=\varnothing\), then both \(x\) and \(y\) are vertices of \(\Gamma_H(M)\), and \(x\leftrightarrow y\).
\end{corollary}


Put \(\displaystyle M^{(1)}=\bigoplus_{\alpha\in{I_1}}S_\alpha\) and \(\displaystyle M^{(\ge2)}=\bigoplus_{\alpha\in{I_{\ge 2}}}E_\alpha\). For \(g\in H\), let \(J_g=\{\alpha\in{I_1}:(S_\alpha)_g\neq0\}\).

\begin{definition}
The multiplicity-one incidence complex is
\[
{\Delta_H^{(1)}}(M) = \{P\subseteq{I_1}:P\subseteq J_g \text{ for some }g\in H\}.
\]
\end{definition}

It is an abstract simplicial complex: \(\varnothing\) belongs to it, and every subset of a member is contained in the same \(J_g\).

\begin{proposition}
\label{proo:4.5}
A nonempty set \(P\subseteq{I_1}\) belongs to \({\Delta_H^{(1)}}(M)\) if and only if \(P=\sigma(x)\) for some nonzero homogeneous \(x\in M\).
\end{proposition}

\begin{proof}
If \(P\in{\Delta_H^{(1)}}(M)\), choose \(g\in H\) with \(P\subseteq J_g\) and nonzero \(x_\alpha\in(S_\alpha)_g\) for \(\alpha\in P\). The element having coordinates \(x_\alpha\) on \(P\) and zero elsewhere is nonzero homogeneuos of degree \(g\), and its obstruction support is exactly \(P\). Conversely, if \(P=\sigma(x)\) for some \(0\neq x\in M_g\), then \(x^{(\alpha)}\in(S_\alpha)_g\setminus\{0\}\) for every \(\alpha\in P\), so \(P\subseteq J_g\) and hence \(P\in{\Delta_H^{(1)}}(M)\).
\end{proof}

For \(P\subseteq{I_1}\), define \({\mathcal{C}}_P=\{0\neq x\in h(M):\sigma(x)=P\}\).

\begin{proposition}
\label{prop:Non-empty}
The class \({\mathcal{C}}_\varnothing\) is nonempty if and only if \({I_{\ge 2}}\neq\varnothing\). In this case \(|{\mathcal{C}}_\varnothing|\ge2\).
\end{proposition}

\begin{proof}
If \({I_{\ge 2}}\neq\varnothing\), choose \(\alpha\in{I_{\ge 2}}\) and \(0\neq s\in h(S_\alpha)\). The two elements \((s,0,\ldots,0)\) and \((0,s,0,\ldots,0)\) in \(E_\alpha\) are distinct members of \({\mathcal{C}}_\varnothing\), so \(|{\mathcal{C}}_\varnothing|\ge2\). Conversely, if \(0\neq x\in{\mathcal{C}}_\varnothing\), then \(A(x)\neq\varnothing\) and \(A(x)\cap{I_1}=\varnothing\); hence every \(\alpha\in A(x)\) satisfies \(m_\alpha\ge2\), and therefore \({I_{\ge 2}}\neq\varnothing\).
\end{proof}

For \(P\subseteq{I_1}\), set \({\mathcal{C}}_P^\Gamma={\mathcal{C}}_P\cap V(\Gamma_H(M))\), and write \({\mathcal{S}}_\Gamma(M)=\{P\subseteq{I_1}: {\mathcal{C}}_P^\Gamma\neq\varnothing\}\).\\

Two disjoint vertex sets \(A,B\subseteq V(G)\) are said to be \emph{completely joined} if every vertex of \(A\) is adjacent to every vertex of \(B\).

The resulting obstruction classes have the following substitution structure.

\begin{theorem}
\label{thrm:structure-theorem}
For the represented vertex classes, the following hold:
\begin{enumerate}
\item \({\mathcal{C}}_\varnothing^\Gamma\), when present, induces a complete graph;
\item every \({\mathcal{C}}_P^\Gamma\) with \(P\in{\mathcal{S}}_\Gamma(M)\setminus\{\varnothing\}\) is independent;
\item distinct classes \({\mathcal{C}}_P^\Gamma,{\mathcal{C}}_Q^\Gamma\) are completely joined if and only if \(P\cap Q=\varnothing\);
\item otherwise there are no edges between them.
\end{enumerate}
\end{theorem}

\begin{proof}
For \(x\in{\mathcal{C}}_P^\Gamma\) and
\(y\in{\mathcal{C}}_Q^\Gamma\), Theorem~\ref{thrm:Main-AT} gives
\[
x\leftrightarrow y \text{ if and only if } P\cap Q=\varnothing.
\]
Thus \({\mathcal{C}}_\varnothing^\Gamma\) is complete, every \({\mathcal{C}}_P^\Gamma\) with \(P\neq\varnothing\) is independent, and two distinct classes are completely joined precisely when their supports are disjoint.
\end{proof}

Accordingly, let \(\mathcal Q(M)\) denote the graph with vertex set \({\mathcal S}_\Gamma(M)\), in which distinct supports \(P,Q\) are adjacent precisely when \(P\cap Q=\varnothing\). We call \(\mathcal Q(M)\) the \emph{obstruction-support quotient} and regard \(P\) as having weight \(|{\mathcal C}_P^\Gamma|\).

Let \(P,Q\in{\mathcal{S}}_\Gamma(M)\) be distinct. Then the classes \({\mathcal{C}}_P^\Gamma\) and \({\mathcal{C}}_Q^\Gamma\) are completely joined if and only if \(P\cap Q=\varnothing\). Thus the quotient is a weighted disjointness graph: the empty-support class, when present, is replaced by a complete graph, while every nonempty support class is replaced by an independent graph of the same cardinality. Multiplicities \(m_\alpha\ge2\) can affect only the class sizes.

\begin{theorem}
\label{thrm:adjacency}
Suppose that \({I_{\ge 2}}\neq\varnothing\). Then every nonzero homogeneous element of \(M\) is a vertex of \(\Gamma_H(M)\); that is, \(Z_H(M)^*=h(M)\setminus\{0\}\). Moreover, every vertex in \({\mathcal{C}}_\varnothing\) is adjacent to every other vertex of \(\Gamma_H(M)\). In particular, \({\mathcal{C}}_\varnothing\) is a universal clique.
\end{theorem}

\begin{proof}
By Proposition~\ref{prop:Non-empty}, \({\mathcal{C}}_\varnothing\) contains at least two elements. Given \(0\neq x\in h(M)\), choose \(u\in{\mathcal{C}}_\varnothing\) with \(u\neq x\). Since \(\sigma(u)=\varnothing\), Corollary~\ref{cor:Main-AT} gives \(u\leftrightarrow x\), so every nonzero homogeneous element is a vertex. Theorem~\ref{thrm:Main-AT} also shows that every \(u\in{\mathcal{C}}_\varnothing\) is adjacent to every other vertex; hence \({\mathcal{C}}_\varnothing\) is a universal clique.
\end{proof}

\begin{corollary}
\label{cor:vertices}
Suppose that \({I_{\ge 2}}=\varnothing\). Then \(Z_H(M)^*=\{\,0\neq x\in h(M):\sigma(x)\neq{I_1}\,\}\).
\end{corollary}

\begin{proof}
Here \({I_1}=\{1,\ldots,t\}\), and every nonzero homogeneous element has nonempty obstruction support. If \(\sigma(x)\neq{I_1}\), choose
\(\alpha\in{I_1}\setminus\sigma(x)\); a nonzero homogeneous element supported on \(S_\alpha\) has support \(\{\alpha\}\) and is adjacent to \(x\). If \(\sigma(x)={I_1}\), then \(\sigma(x)\) meets the support of every nonzero homogeneous element, so \(x\) has no neighbor. The result follows from Theorem~\ref{thrm:Main-AT}.
\end{proof}

\begin{corollary}
Let \(M=S^m\), where \(S\) is gr-simple and \(m\ge2\). Then \(\Gamma_H(M)\) is complete.
\end{corollary}

\begin{proof}
This follows from Theorems~\ref{thrm:adjacency} and~\ref{thrm:Main-AT}, since every vertex has empty obstruction support.
\end{proof}

\begin{corollary}
\label{thrm:diam<3}
Suppose that \({I_{\ge 2}}\neq\varnothing\). Then \(\Gamma_H(M)\) is connected and \( \diam\Gamma_H(M)\le2\). More precisely, its diameter is \(1\) when the graph is complete, and \(2\) otherwise.
\end{corollary}

\begin{proof}
By Theorem~\ref{thrm:adjacency}, a vertex in \({\mathcal{C}}_\varnothing\) is universal, so every two vertices are at distance at most \(2\). The remaining assertion is immediate.
\end{proof}

\begin{example}
Let \(R=\mathbb Z\) and \(M=\mathbb Z_2\oplus\mathbb Z_2\) be with a trivial grading. There is one simple isomorphism type of multiplicity two, so \({I_1}=\varnothing\). Hence, every nonzero element has empty obstruction support and \(\Gamma_H(M)=K_3\).
\end{example}

\begin{corollary}
Let \(M=S_1\oplus S_2\), where \(S_1\) and \(S_2\) are gr-simple and \(\Ann_R(S_1)\neq\Ann_R(S_2)\). Then \(\Gamma_H(M)\) is complete bipartite.
\end{corollary}

\begin{proof}
By Corollary~\ref{cor:vertices}, the only vertex supports are \(\{1\}\) and \(\{2\}\), and Theorem~\ref{thrm:structure-theorem} concludes the result.
\end{proof}

\begin{example}
Let \(R=\mathbb Z\) and \(M=\mathbb Z_2\oplus\mathbb Z_3\) with trivial grading. Both blocks have multiplicity one. Full-support elements have no zero-divisor neighbor and therefore do not belong to \(V(\Gamma_H(M))\). The surviving class sizes are \(|{\mathcal{C}}_{\{1\}}|=1\) and \(|{\mathcal{C}}_{\{2\}}|=2\). Hence \(\Gamma_H(M)=K_{1,2}\).
\end{example}

\begin{example}
Let \(R=\mathbb Z\) and \(M=(\mathbb Z_2\oplus\mathbb Z_2)\oplus\mathbb Z_3\) with trivial grading. The \(\mathbb Z_2\)-type has multiplicity two, while the \(\mathbb Z_3\)-type has multiplicity one, so \({I_1}=\{1\}\). The empty obstruction class has size \(3\), whereas \(|{\mathcal{C}}_{\{1\}}|=8\). Hence, \(\Gamma_H(M)=K_3\vee\overline{K}_8\).
The corresponding decomposition into the two obstruction classes is illustrated in Figure~\ref{fig:mixed}.

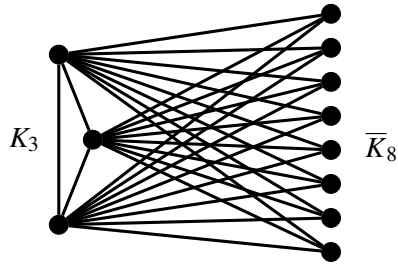
\begin{figure}[ht]
\centering
\begin{tikzpicture}[scale=0.9]

\node[vertex] (a1) at (-2,0.5) {};
\node[vertex] (a2) at (-1.5,-0.75) {};
\node[vertex] (a3) at (-2,-2) {};

\node[vertex] (b1) at (2,1.10) {};
\node[vertex] (b2) at (2,0.60) {};
\node[vertex] (b3) at (2,0.10) {};
\node[vertex] (b4) at (2,-0.40) {};
\node[vertex] (b5) at (2,-0.90) {};
\node[vertex] (b6) at (2,-1.40) {};
\node[vertex] (b7) at (2,-1.90) {};
\node[vertex] (b8) at (2,-2.40) {};

\draw[edge] (a1)--(a2);
\draw[edge] (a2)--(a3);
\draw[edge] (a3)--(a1);

\foreach \a in {a1,a2,a3}{
    \foreach \b in {b1,b2,b3,b4,b5,b6,b7,b8}{
        \draw[edge] (\a)--(\b);
    }
}
\node[graphlabel] at (-2.5,-0.75) {$K_3$};
\node[graphlabel] at (2.75,-0.85) {$\overline{K}_8$};

\end{tikzpicture}
\caption{The mixed graph \(K_3\vee\overline{K}_8\).}
\label{fig:mixed}
\end{figure}
\end{example}

\section{Graph Structure and Invariants}
\subsection{Clique number and coloring}

Assume throughout this subsection that \(\Gamma_H(M)\) is finite. Recall that \(s=|{I_1}|\), and let \(u=|{\mathcal{C}}_\varnothing|\) when \({I_{\ge 2}}\neq\varnothing\); otherwise set \(u=0\).

Because cliques correspond to pairwise disjoint nonempty supports, while the empty class is universal, both the clique number and the chromatic number can be read off directly.

\begin{theorem}
\label{thrm:clique=chromatic}
Let \(s=|{I_1}|\), and let \(u=|{\mathcal{C}}_\varnothing|\) when \({I_{\ge 2}}\neq\varnothing\), with \(u=0\) otherwise. If \(M\) consists of a single multiplicity-one simple component, then \(\Gamma_H(M)\) is empty. In every other case,
\[
\omega(\Gamma_H(M)) = \chi(\Gamma_H(M)) = u+s.
\]
\end{theorem}

\begin{proof}
For each \(\alpha\in{I_1}\), the singleton support \(\{\alpha\}\) is realized by a nonzero homogeneous element of \(S_\alpha\). Except in the excluded one-component case, the corresponding vertex class is nonempty. Choosing one vertex from each singleton class, together with all \(u\) vertices of \({\mathcal{C}}_\varnothing\) when this class is present, gives a clique of size \(u+s\).

Conversely, a clique contains at most one vertex from each nonempty obstruction class, and its nonempty supports are pairwise disjoint. Since these are nonempty subsets of the \(s\)-element set \({I_1}\), at most \(s\) of them can occur. The empty-support class contributes at most \(u\) further vertices. Hence
\[
\omega(\Gamma_H(M))=u+s.
\]

For the coloring, use distinct colors on the vertices of \({\mathcal{C}}_\varnothing\), and index \(s\) further colors by \({I_1}\). For each represented nonempty support \(P\), choose \(\alpha_P\in P\) and color every vertex of \({\mathcal{C}}_P^\Gamma\) by \(\alpha_P\). If two such classes are joined, their supports are disjoint, so they receive different colors. Thus
\[
\chi(\Gamma_H(M))\le u+s.
\]
Since \(\omega(\Gamma_H(M))\le\chi(\Gamma_H(M))\), equality follows.
\end{proof}

\begin{corollary}
Suppose that \(M\) is multiplicity-free with \(t\ge2\) simple components, and that \(\Gamma_H(M)\) is finite. Then \(\omega(\Gamma_H(M))=\chi(\Gamma_H(M))=t\).
\end{corollary}

The equality in Theorem~\ref{thrm:clique=chromatic} should not be interpreted as saying that \(\Gamma_H(M)\) is perfect. The following example makes this distinction clear.

\begin{example}
Let \(R=\mathbb Z, ~M=\mathbb Z_2\oplus\mathbb Z_3\oplus\mathbb Z_5 \oplus\mathbb Z_7\oplus\mathbb Z_{11}\), with the trivial grading. Each simple component occurs with multiplicity one, so \({I_1}=\{1,2,3,4,5\}\). By Theorem~\ref{thrm:clique=chromatic}, \(\omega(\Gamma_H(M))=\chi(\Gamma_H(M))=5\).

Now consider
\[
\begin{aligned}
x_1&=(1,1,0,0,0),&
x_2&=(0,0,1,1,0),&
x_3&=(1,0,0,0,1),\\
x_4&=(0,1,1,0,0),&
x_5&=(0,0,0,1,1).
\end{aligned}
\]
Their obstruction supports are
\[
\sigma(x_1)=\{1,2\},~
\sigma(x_2)=\{3,4\},~
\sigma(x_3)=\{1,5\},~
\sigma(x_4)=\{2,3\},~
\sigma(x_5)=\{4,5\}.
\]

Consecutive supports in this list are disjoint, with the indices read cyclically. Hence
\(x_1,x_2,x_3,x_4,x_5\) form a \(5\)-cycle. Moreover, every pair of nonconsecutive supports has a common element, so there are no additional edges among these five vertices. Thus they induce a copy of \(C_5\) in \(\Gamma_H(M)\).

Since \(\omega(C_5)=2\) and \(\chi(C_5)=3\), this induced subgraph is not perfect. Therefore \(\Gamma_H(M)\) is not perfect, even though
\(\omega(\Gamma_H(M))=\chi(\Gamma_H(M))\). This shows that the equality established in Theorem~\ref{thrm:clique=chromatic} applies to the graph itself and does not imply perfectness.
\end{example}

\subsection{Metric and girth structure}

\begin{theorem}
\label{thrm:distance}
Assume that \({I_{\ge 2}}=\varnothing\). Let \(x\) and \(y\) be distinct vertices of \(\Gamma_H(M)\), and set \(A=\sigma(x)\) and \(B=\sigma(y)\). Then
\[
d(x,y)=
\begin{cases}
1, & A\cap B=\varnothing,\\
2, & A\cap B\neq\varnothing \text{ and } A\cup B\neq{I_1},\\
3, & A\cap B\neq\varnothing \text{ and } A\cup B={I_1}.
\end{cases}
\]
\end{theorem}

\begin{proof}
Since \({I_{\ge 2}}=\varnothing\), every simple type occurs with multiplicity one. Thus, by Corollary~\ref{cor:vertices}, the obstruction support of a vertex is a nonempty proper subset of \({I_1}\). In particular, \(A\) and \(B\) are nonempty and proper.

If \(A\cap B=\varnothing\), then \(x\) and \(y\) are adjacent by Theorem~\ref{thrm:Main-AT}. Hence \(d(x,y)=1\).

Now suppose that \(A\cap B\neq\varnothing\) and \(A\cup B\neq{I_1}\). Choose \(\alpha\in{I_1}\setminus(A\cup B)\). A nonzero homogeneous element supported only on \(S_\alpha\) gives a vertex \(z\) with \(\sigma(z)=\{\alpha\}\). Since \(\alpha\notin A\cup B\), the support of \(z\) is disjoint from both \(A\) and \(B\). Hence \(x\leftrightarrow z\leftrightarrow y\). On the other hand, \(A\cap B\neq\varnothing\), so \(x\) and \(y\) are not adjacent. Therefore \(d(x,y)=2\).

It remains to consider the case \(A\cap B\neq\varnothing\) and \(A\cup B={I_1}\). There is no common neighbor of \(x\) and \(y\). Indeed, if \(z\) were adjacent to both, then Theorem~\ref{thrm:Main-AT} would give \(\sigma(z)\cap A=\sigma(z)\cap B=\varnothing\). Hence \(\sigma(z)\cap(A\cup B)=\varnothing\), and so \(\sigma(z)=\varnothing\), contrary to the fact that every vertex has nonempty obstruction support. Thus \(d(x,y)>2\).

Since \(A\) and \(B\) are proper subsets of \({I_1}\) and \(A\cup B={I_1}\), both \(A\setminus B\) and \(B\setminus A\) are nonempty. Choose \(\alpha\in A\setminus B\) and \(\beta\in B\setminus A\). Take vertices \(u\) and \(v\) with \(\sigma(u)=\{\beta\}\) and \(\sigma(v)=\{\alpha\}\). Then \(A\cap\sigma(u)=\varnothing\), \(\sigma(u)\cap\sigma(v)=\varnothing\), and \(\sigma(v)\cap B=\varnothing\). Hence \(x-u-v-y\) is a path of length \(3\). Consequently \(d(x,y)=3\).
\end{proof}

\begin{theorem}
\label{thrm:girth}
Assume that \(M\neq0\), the following hold.
\begin{enumerate}
\item If \({I_{\ge 2}}\neq\varnothing\), then \( g(\Gamma_H(M))=3. \)

\item If \({I_{\ge 2}}=\varnothing\) and \(s\ge3\), then \( g(\Gamma_H(M))=3. \)

\item If \({I_{\ge 2}}=\varnothing\) and \(s=2\), then \( \Gamma_H(M)\cong K_{a,b}, \) where \(a=|{\mathcal{C}}_{\{1\}}| \text{ and } b=|{\mathcal{C}}_{\{2\}}|. \) Consequently,
\[
g(\Gamma_H(M))=
\begin{cases}
4, & a,b\ge2,\\
\infty, & \text{otherwise}.
\end{cases}
\]

\item If \({I_{\ge 2}}=\varnothing\) and \(s=1\), then \(\Gamma_H(M)\) is empty.
\end{enumerate}
\end{theorem}

\begin{proof} We consider the four cases separately.
\begin{enumerate}
\item A repeated block contains three distinct homogeneous vertices with empty obstruction support, which form a triangle by Theorem~\ref{thrm:Main-AT}.
\item Three distinct singleton supports are pairwise disjoint, so Corollary~\ref{cor:Main-AT} gives a triangle.
\item If \(s=2\), Corollary~\ref{cor:vertices} shows that the only vertex supports are \(\{1\}\) and \(\{2\}\). Theorem~\ref{thrm:structure-theorem} therefore gives \(\Gamma_H(M)\cong K_{a,b}\). If \(a,b\ge2\), the graph has a \(4\)-cycle and no shorter cycle; otherwise it is a star and hence acyclic.
\item If \(s=1\), Corollary~\ref{cor:vertices} gives \(V(\Gamma_H(M))=\varnothing\).
\end{enumerate}
\end{proof}

\begin{example}
Let \(M=\mathbb Z_2\oplus\mathbb Z_3\oplus\mathbb Z_5\) with trivial grading. The nonempty obstruction supports represented by vertices are \(\{1\},\{2\},\{3\},\{1,2\},\{1,3\},\{2,3\}\).

At the quotient level the three singleton supports form a triangle. A two-element support has only one disjoint nonempty support, namely its complementary singleton.

\begin{figure}[ht]
\centering
\begin{tikzpicture}[scale=1.0]

\node[vertex] (s1) at (0,1.3) {};
\node[vertex] (s2) at (-1.3,-0.7) {};
\node[vertex] (s3) at (1.3,-0.7) {};

\node[vertex] (p23) at (0,2.45) {};
\node[vertex] (p13) at (-2.35,-1.45) {};
\node[vertex] (p12) at (2.35,-1.45) {};

\draw[edge] (s1)--(s2);
\draw[edge] (s2)--(s3);
\draw[edge] (s3)--(s1);

\draw[edge] (s1)--(p23);
\draw[edge] (s2)--(p13);
\draw[edge] (s3)--(p12);

\node[graphlabel,right=3pt of s1] {$\{1\}$};
\node[graphlabel,above left=1pt of s2] {$\{2\}$};
\node[graphlabel,above right=1pt of s3] {$\{3\}$};

\node[graphlabel,above=3pt of p23] {$\{2,3\}$};
\node[graphlabel,left=3pt of p13] {$\{1,3\}$};
\node[graphlabel,right=3pt of p12] {$\{1,2\}$};

\end{tikzpicture}
\caption{The obstruction-support quotient for three multiplicity-one types.}
\label{fig:three}
\end{figure}
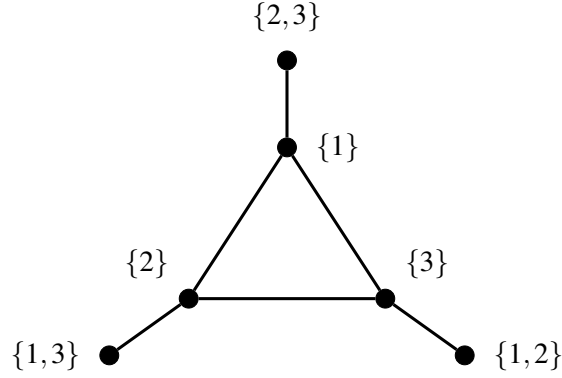

In particular, \(\omega(\Gamma_H(M))=\chi(\Gamma_H(M))=3\) and the girth is \(3\).
\end{example}

\subsection{Reconstruction from twin classes}

Two distinct vertices \(u,v\in V(G)\) are called \emph{false twins} if \(N_G(u)=N_G(v)\), and \emph{true twins} if \(N_G[u]=N_G[v]\).

\begin{proposition}
Among vertices with nonempty obstruction support,
\[
N_{\Gamma_H(M)}(x)=N_{\Gamma_H(M)}(y) \text{ if and only if } \sigma(x)=\sigma(y).
\]
When the empty-support class is present, its vertices are true twins.
\end{proposition}

\begin{proof}
If \(\sigma(x)=\sigma(y)\neq\varnothing\), then \(x\) and \(y\) are nonadjacent and, by Theorem~\ref{thrm:Main-AT}, have the same neighbors; hence their open neighborhoods are equal.

Conversely, if \(\sigma(x)\neq\sigma(y)\), choose \(\alpha\in\sigma(x)\setminus\sigma(y)\) after interchanging \(x\) and \(y\) if necessary, and let \(z\) have support \(\{\alpha\}\). Then \(z\) is adjacent to \(y\) but not to \(x\), so their neighborhoods differ.

If \(\sigma(x)=\sigma(y)=\varnothing\), Theorem~\ref{thrm:adjacency} makes both vertices universal, hence true twins.
\end{proof}

Thus, the nonempty obstruction classes are exactly the false-twin classes, except that universal vertices with nonempty support require separate consideration.

\begin{proposition}
\label{prop:5.8}
Let \(P\in{\mathcal{S}}_\Gamma(M)\setminus\{\varnothing\}\). The class \({\mathcal{C}}_P^\Gamma\) contains a universal vertex if and only if there is an \(\alpha\in{I_1}\) such that
\[
P=\{\alpha\},~|{\mathcal{C}}_P^\Gamma|=1,
\]
and \(\alpha\notin Q\) for every \(Q\in{\mathcal{S}}_\Gamma(M)\setminus\{P\}\). In that case the unique vertex of \({\mathcal{C}}_P^\Gamma\) is universal.
\end{proposition}

\begin{proof}
Suppose that \(x\in{\mathcal{C}}_P^\Gamma\) is universal. Since \(P\neq\varnothing\), the class \({\mathcal{C}}_P^\Gamma\) is independent by Theorem~\ref{thrm:structure-theorem}. Hence it cannot contain a second vertex, and so \(|{\mathcal{C}}_P^\Gamma|=1\).

We claim that \(P\) is a singleton. If \(|P|\ge2\), choose distinct \(\alpha,\beta\in P\), together with nonzero homogeneous elements supported only on \(S_\alpha\) and \(S_\beta\). Their obstruction supports are \(\{\alpha\}\) and \(\{\beta\}\), so Corollary~\ref{cor:Main-AT} shows that both are vertices. In particular, \({\mathcal{C}}_{\{\alpha\}}^\Gamma\neq\varnothing\). Since \(\{\alpha\}\cap P\neq\varnothing\), no vertex of \({\mathcal{C}}_{\{\alpha\}}^\Gamma\) is adjacent to \(x\), a contradiction. Thus \(P=\{\alpha\}\) for some \(\alpha\in{I_1}\).

If \(Q\in{\mathcal{S}}_\Gamma(M)\setminus\{P\}\) contained \(\alpha\), then Theorem~\ref{thrm:structure-theorem} would give no edges between \({\mathcal{C}}_P^\Gamma\) and \({\mathcal{C}}_Q^\Gamma\), again contradicting the universality of \(x\). Hence \(\alpha\notin Q\) for every such \(Q\).

Conversely, suppose that the stated conditions hold, and let \(x\) be the unique vertex of \({\mathcal{C}}_{\{\alpha\}}^\Gamma\). There is no other vertex in its own class. Every other represented class has support disjoint from \(\{\alpha\}\), so Theorem~\ref{thrm:structure-theorem} shows that \(x\) is adjacent to every other vertex of \(\Gamma_H(M)\). Thus \(x\) is universal.
\end{proof}

Thus the abstract graph always recovers all nonuniversal nonempty obstruction classes, their weights, and their mutual disjointness relation. The unresolved information is confined to universal vertices: a universal vertex may come either from the empty-support class or from a singleton class described in Proposition~\ref{prop:5.8}.

Let \(U(G)\) denote the set of universal vertices of \(G\). To remove this ambiguity in the full reconstruction theorem, we assume
\begin{equation}
\label{eq:(U)}
U(\Gamma_H(M))={\mathcal{C}}_\varnothing^\Gamma.
\tag{U}
\end{equation}
\begin{theorem}[Graph Reconstruction]
\label{thrm:graph-reconstruction}
Assume condition~\eqref{eq:(U)}, and exclude the empty graph arising from a single multiplicity-one gr-simple module. Then the abstract graph \(\Gamma_H(M)\) determines, up to a permutation of \({I_1}\),
\begin{enumerate}
\item the family \({\mathcal{S}}_\Gamma(M)\) of obstruction supports represented by vertices;
\item the cardinality of each represented obstruction class;
\item the disjointness relation on \({\mathcal{S}}_\Gamma(M)\);
\item the number \(s=|{I_1}|\);
\item whether the empty-support class is present and, when it is present, its cardinality.
\end{enumerate}
\end{theorem}

\begin{proof}
Write \(\Gamma=\Gamma_H(M)\). By condition~\eqref{eq:(U)},
\({\mathcal{C}}_\varnothing^\Gamma=U(\Gamma)\), so the graph determines whether the empty-support class occurs and, if it does, its cardinality.

On the nonuniversal vertices define
\[
x\sim y \text{ if and only if } N_\Gamma(x)=N_\Gamma(y).
\]
The adjacency criterion shows that the \(\sim\)-classes are precisely the represented nonempty obstruction classes. Therefore, their cardinalities give the corresponding class weights.

Let \(\mathcal Q_\Gamma\) have these classes as vertices, with two distinct classes adjacent when they are completely joined in \(\Gamma\). Then \(\mathcal Q_\Gamma\) determines the disjointness relation since according to Theorem~\ref{thrm:structure-theorem}, the classes corresponding to \(P\) and \(Q\) are adjacent exactly when \(P\cap Q=\varnothing\). 

If \(\mathcal Q_\Gamma\) is empty, then \(s=0\); otherwise a singleton support would give a nonuniversal class, apart from the excluded one-simple-module case. Suppose therefore that \(\mathcal Q_\Gamma\neq\varnothing\). The \(s\) singleton supports give a clique of size \(s\). Conversely, every clique corresponds to pairwise disjoint nonempty subsets of the \(s\)-element set \({I_1}\), and hence has at most \(s\) vertices. Thus, \(\omega(\mathcal Q_\Gamma)=s\). A clique of size \(s\) must consist of singleton supports, so these classes are determined up to a permutation of \({I_1}\).

For a represented nonempty support \(P\) and \(\alpha\in{I_1}\),
\[
\alpha\in P \text{ if and only if } {\mathcal{C}}_P^\Gamma={\mathcal{C}}_{\{\alpha\}}^\Gamma
 \text{ or } {\mathcal{C}}_P^\Gamma \text{ is not adjacent to } {\mathcal{C}}_{\{\alpha\}}^\Gamma \text{ in }\mathcal Q_\Gamma.
\]
For distinct classes this means \(P\cap\{\alpha\}\neq\varnothing\). Hence all represented supports, their weights and their disjointness relation are recovered, together with \(s\) and the empty-support class, up to relabeling of \({I_1}\).
\end{proof}

The following example shows why condition~\eqref{eq:(U)} is needed in the reconstruction theorem.

\begin{example}
\label{ex:condition-u}
Let \(k=\mathbb F_2\). First take \(M=k^2\) with the trivial grading. Writing \(M=S^2\) with \(S=k\), we have \(s=0\). The three nonzero elements have empty obstruction support and are pairwise adjacent, so \(\Gamma_H(M)\cong K_3\).

Now let \(R=k\times k\times k\) with the trivial \(C_3\)-grading, and let \(S_1,S_2,S_3\) be the three simple coordinate \(R\)-modules, graded in three distinct degrees of \(C_3\). Put \(M'=S_1\oplus S_2\oplus S_3\). The only nonzero homogeneous elements of \(M'\) are those supported on a single summand. Their obstruction supports are \(\{1\}\), \(\{2\}\), and \(\{3\}\), so they are pairwise adjacent. Hence, \(\Gamma_H(M')\cong K_3\).

Thus \(\Gamma_H(M)\cong\Gamma_H(M')\), but \(s=0\) for \(M\) and \(s=3\) for \(M'\). In the second case, the universal vertices have nonempty obstruction support. This is the ambiguity excluded by condition~\eqref{eq:(U)}.
\end{example}

\begin{corollary}
\label{cor:incidence-recovery}
Assume condition~\eqref{eq:(U)}. Then the following hold.
\begin{enumerate}
\item If \({I_{\ge 2}}\neq\varnothing\), then \(\Gamma_H(M)\) determines \({\Delta_H^{(1)}}(M)\) completely, up to a permutation of \({I_1}\).

\item If \({I_{\ge 2}}=\varnothing\) and \(s\ge2\), then \(\Gamma_H(M)\) determines every proper nonempty face of \({\Delta_H^{(1)}}(M)\), up to relabeling. In general, the graph does not determine whether the full face \({I_1}\) belongs to \({\Delta_H^{(1)}}(M)\).
\end{enumerate}
\end{corollary}

\begin{proof}
This follows from Proposition~\ref{proo:4.5},
Theorem~\ref{thrm:adjacency}, Corollary~\ref{cor:vertices},
and Theorem~\ref{thrm:graph-reconstruction}.
\end{proof}

The qualification in the second part of Corollary~\ref{cor:incidence-recovery} is needed.

\begin{example}
Let \(H=C_2\times C_2\), written additively, and set
\(L_1=\{(0,0),(1,0)\}\), \(L_2=\{(0,0),(0,1)\}\), and \(L_3=\{(0,0),(1,1)\}\). For \(i=1,2,3\), let \(K_i=\mathbb F_2[L_i]\) with its natural \(H\)-grading, and put \(R=K_1\times K_2\times K_3\). Consider \(S_i=K_i\) as a graded \(R\)-module through the \(i\)-th coordinate projection. Each \(S_i\) is gr-simple. Let \(M=S_1\oplus S_2\oplus S_3\). Since, \(L_1\cap L_2\cap L_3=\{(0,0)\}\), we have \(\{1,2,3\}\in{\Delta_H^{(1)}}(M)\). Now, shift the grading of \(S_3\) by \((1,0)\), and write \(S_3'\) for the shifted module. Its degree support is \((1,0)+L_3=\{(1,0),(0,1)\}\). Put \(M'=S_1\oplus S_2 \oplus S_3'\). The three pairwise intersections remain nonempty, but the triple intersection is empty. Hence, \(\{1,2,3\}\notin{\Delta_H^{(1)}}(M')\).

For both \(M\) and \(M'\), each singleton obstruction class has size \(2\), while each two-element class has size \(1\). Thus, the proper nonempty obstruction classes have the same weights and disjointness relations. By Theorem~\ref{thrm:structure-theorem}, \(\Gamma_H(M)\cong\Gamma_H(M')\), although the two incidence complexes differ in the full face.
\end{example}
\subsection{Kneser layers}

The quotient graph need not be complete multipartite. Fixed-size support layers already exhibit the Kneser-graph structure. For \(1\le k<s\), define
\[
\Delta_{H,k}^{(1)}(M) = \{P\in{\Delta_H^{(1)}}(M):|P|=k\}.
\]
Recall that the Kneser graph \(KG(s,k)\) has as its vertices the \(k\)-element subsets of an \(s\)-element set, with two vertices adjacent exactly when the corresponding subsets are disjoint. This gives the following Kneser-layer description.

\begin{theorem}
\label{thrm;5-kneser}
Let \(1\le k<s\). If
\[
\Delta_{H,k}^{(1)}(M)=\binom{{I_1}}{k},
\]
then the quotient graph induced by the \(k\)-element obstruction classes is isomorphic to \(KG(s,k)\). Moreover, the corresponding induced subgraph of \(\Gamma_H(M)\) is obtained from \(KG(s,k)\) by replacing each vertex \(P\) with the independent set \({\mathcal{C}}_P^\Gamma\).
\end{theorem}

\begin{proof}
Since \(k<s\), every \(k\)-element face is a proper support and therefore represents a vertex class. By Theorem~\ref{thrm:structure-theorem}, two such classes are completely joined exactly when their supports are disjoint. Hence the quotient on the \(k\)-element classes is \(KG(s,k)\). Since every such support is nonempty, each corresponding obstruction class is independent, giving the asserted independent-set blow-up.
\end{proof}

\begin{corollary}
\label{cor:kneser-coloring}
Suppose that \(\Delta_{H,k}^{(1)}(M)=\binom{{I_1}}{k}\) and \(s\ge2k\). Then the induced \(k\)-support layer of \(\Gamma_H(M)\) has chromatic number \(s-2k+2\).
\end{corollary}

\begin{proof}
By Theorem~\ref{thrm;5-kneser}, the layer has the same chromatic number as \(KG(s,k)\), so Lovász's theorem \cite{Lovasz1978} gives \(\chi=s-2k+2\).
\end{proof}

\begin{corollary}
\label{cor:not-perfect}
Suppose that \(J\subseteq{I_1}\) has five elements and that every two-element subset of \(J\) belongs to \({\Delta_H^{(1)}}(M)\). Then \(\Gamma_H(M)\) contains an induced \(C_5\). In particular, \(\Gamma_H(M)\) is not perfect.
\end{corollary}

\begin{proof}
After relabeling the elements of \(J\), write \(J=\{1,2,3,4,5\}\). Consider the five supports
\[
P_1=\{1,2\},~
P_2=\{3,4\},~
P_3=\{1,5\},~
P_4=\{2,3\},~
P_5=\{4,5\}.
\]
By hypothesis, each \(P_i\) is realized, and since \(|P_i|<s\), the corresponding obstruction class contains a vertex; choose \(x_i\in{\mathcal{C}}_{P_i}^\Gamma\). The disjointness graph of \(P_1,\ldots,P_5\) is exactly a \(5\)-cycle: consecutive supports are disjoint and nonconsecutive supports intersect. Hence, Theorem~\ref{thrm:Main-AT} shows that \(x_1,\ldots,x_5\) induce \(C_5\).

Since \(C_5\) is not perfect and every induced subgraph of a perfect graph is perfect, it follows that \(\Gamma_H(M)\) is not perfect.
\end{proof}

\section{Finite Modules, Class Weights and Spectrum}

Throughout this section we impose the stronger hypothesis that \(M\) is finite as a set. This is stronger than finite generation and is used here only to make the class counts and adjacency spectrum finite.

\subsection{Class weights and enumeration}

Write \(B=M^{(\ge2)}\). For \(g\in H\), let \(b_g=|B_g|\). If \({I_{\ge 2}}=\varnothing\), we adopt the convention \(B_g=\{0\}\), so that \(b_g=1\). For each \(\alpha\in{I_1}\), put \(q_{\alpha,g}=|(S_\alpha)_g|\).

\begin{theorem}
\label{thrm:SC-enumeration}
Let \(P\in{\Delta_H^{(1)}}(M)\) be nonempty. Then
\[
|{\mathcal{C}}_P| = \sum_{g\in H} b_g\prod_{\alpha\in P}(q_{\alpha,g}-1).
\]
Moreover,
\[
|{\mathcal{C}}_\varnothing| = \sum_{g\in H}(b_g-1).
\]
\end{theorem}

\begin{proof}
Fix a nonempty \(P\in{\Delta_H^{(1)}}(M)\), and count the elements of \({\mathcal{C}}_P\) one homogeneous degree at a time. For \(g\in H\), an element \(x\in M_g\) has \(\sigma(x)=P\) precisely when its \(\alpha\)-th multiplicity-one coordinate is nonzero for every \(\alpha\in P\), and zero for every \(\alpha\in{I_1}\setminus P\). Hence each \(\alpha\in P\) contributes \(q_{\alpha,g}-1\) choices. The component of \(x\) in \(B_g\) is unrestricted and contributes \(b_g\) choices. Thus the number of degree-\(g\) elements with obstruction support \(P\) is
\[
b_g\prod_{\alpha\in P}(q_{\alpha,g}-1).
\]
Since \(P\neq\varnothing\), all such elements are nonzero. A nonzero homogeneous element has a unique degree, so these degree-by-degree counts are disjoint. Summing over the homogeneous degrees gives
\[
|{\mathcal{C}}_P|
=
\sum_{g\in H}
b_g\prod_{\alpha\in P}(q_{\alpha,g}-1).
\]

If \(P=\varnothing\), every multiplicity-one coordinate must vanish. Thus \({\mathcal{C}}_\varnothing\) consists exactly of the nonzero homogeneous elements of \(B\). There are \(b_g-1\) such elements in degree \(g\), and therefore
\[
|{\mathcal{C}}_\varnothing| = \sum_{g\in H}(b_g-1).
\]
Since \(M\) is finite, only finitely many terms in these sums are nonzero.
\end{proof}

\noindent{\bf Note:} In the trivially graded case, the formulas simplify considerably. Writing
\(b=|M^{(\ge2)}|\) and \(q_\alpha=|S_\alpha|\) for
\(\alpha\in{I_1}\), we obtain
\[
|{\mathcal{C}}_\varnothing|=b-1,~|{\mathcal{C}}_P| = b\prod_{\alpha\in P}(q_\alpha-1)
\]
for every nonempty \(P\in{\Delta_H^{(1)}}(M)\).

\begin{lemma}
\label{lemma:sets-equality}
If \(P\in{\mathcal{S}}_\Gamma(M)\), then \({\mathcal{C}}_P^\Gamma={\mathcal{C}}_P\).
\end{lemma}

\begin{proof}
Choose \(x\in{\mathcal{C}}_P^\Gamma\) with a neighbor \(y\); then Theorem~\ref{thrm:Main-AT} gives \(P\cap\sigma(y)=\varnothing\). Every \(z\in{\mathcal{C}}_P\) distinct from \(y\) is therefore adjacent to \(y\), so every element of \({\mathcal{C}}_P\) is a vertex.
\end{proof}

\begin{proposition}
\label{prop:degree}
Let \(P\in{\mathcal{S}}_\Gamma(M)\) and \(x\in{\mathcal{C}}_P\). Then
\[
\deg(x)=
\begin{cases}
|V(\Gamma_H(M))|-1,
& P=\varnothing,\\[1mm]
\displaystyle
\sum_{\substack{Q\in{\mathcal{S}}_\Gamma(M)\\
P\cap Q=\varnothing}}
|{\mathcal{C}}_Q|,
& P\neq\varnothing.
\end{cases}
\]
\end{proposition}

\begin{proof}
If \(P=\varnothing\), Theorem~\ref{thrm:Main-AT} makes \(x\) universal, so \(\deg(x)=|V(\Gamma_H(M))|-1\). If \(P\neq\varnothing\), Theorem~\ref{thrm:structure-theorem} shows that the neighbors of \(x\) are exactly the vertices in the classes \({\mathcal{C}}_Q\) with \(P\cap Q = \varnothing\).
\end{proof}

Vertices with the same obstruction support have the same degree. If \(d_P\) denotes this common degree for \(P\in{\mathcal{S}}_\Gamma(M)\), then
\[
2|E(\Gamma_H(M))| = \sum_{P\in{\mathcal{S}}_\Gamma(M)} |{\mathcal{C}}_P|\,d_P.
\]

\subsection{Equitable partition and spectrum}

Since \(M\) is finite, \(\Gamma_H(M)\) is finite. Assume throughout this subsection that \(V(\Gamma_H(M))\neq\varnothing\). We use equitable partitions and quotient matrices in the usual sense; see \cite{BrouwerHaemers2012,GodsilRoyle2001}.

A partition \(V(G)=V_1\mathbin{\dot\cup}\cdots\mathbin{\dot\cup}V_r\) is \emph{equitable} if, for every \(i,j\), each vertex of \(V_i\) has the same number \(q_{ij}\) of neighbors in \(V_j\). The matrix \(Q=(q_{ij})\) is the corresponding quotient matrix.

Write \({\mathcal{S}}_\Gamma(M)=\{P_1,\ldots,P_r\}\) and \(c_i=|{\mathcal{C}}_{P_i}|\). By Lemma~\ref{lemma:sets-equality}, the classes \({\mathcal{C}}_{P_1},\ldots,{\mathcal{C}}_{P_r}\) partition \(V(\Gamma_H(M))\), and Theorem~\ref{thrm:structure-theorem} shows that this partition is equitable.

Let \(Q=(Q_{ij})\) be the corresponding quotient matrix. Then
\[
Q_{ij}=
\begin{cases}
c_j, & i\neq j \text{ and } P_i\cap P_j=\varnothing,\\
c_i-1, & i=j \text{ and } P_i=\varnothing,\\
0, & \text{otherwise}.
\end{cases}
\]

If \(D=\operatorname{diag}(c_1,\ldots,c_r)\), then \(D^{1/2}QD^{-1/2}\) is symmetric, so the eigenvalues of \(Q\) are real. This quotient matrix gives the spectral reduction below.

\begin{theorem}
\label{thrm:spectrum}
Let \(Q\) be the quotient matrix defined above. Then every eigenvalue of \(Q\) is an eigenvalue of \(\Gamma_H(M)\). More precisely, as a multiset,
\[
\Spec(\Gamma_H(M))=\Spec(Q)\uplus\{-1^{\,\mu_{-1}}\}\uplus\{0^{\,\mu_0}\}.
\]
Here \(\mu_{-1}=|{\mathcal{C}}_\varnothing|-1\) when \({\mathcal{C}}_\varnothing\neq\varnothing\), and \(\mu_{-1}=0\) otherwise, while
\(\mu_0=\sum_{\substack{P\in{\mathcal{S}}_\Gamma(M)\\P\neq\varnothing}}(|{\mathcal{C}}_P|-1)\).

Equivalently, if \(A\) is the adjacency matrix of \(\Gamma_H(M)\), then
\[
\det(\lambda I-A)=\det(\lambda I-Q (\lambda+1)^{\mu_{-1}}\lambda^{\mu_0}.
\]
Here \(\uplus\) denotes multiset union. If \(0\) or \(-1\) is already an eigenvalue of \(Q\), its multiplicity is increased by \(\mu_0\) or \(\mu_{-1}\), respectively.
\end{theorem}

\begin{proof}
Let \(\Gamma=\Gamma_H(M)\), and let \(A\) be its adjacency matrix. For \(1\le i\le r\), denote by \(\mathbf 1_i\) the characteristic vector of \({\mathcal{C}}_{P_i}\), and let \(W=\operatorname{span}\{\mathbf 1_1,\ldots,\mathbf 1_r\}\).

Since the partition is equitable, \(W\) is invariant under \(A\). For each \(j\), the vector \(A\mathbf 1_j\) is constant on every class \({\mathcal{C}}_{P_i}\), with value \(Q_{ij}\). Hence \(\displaystyle A\mathbf 1_j=\sum_{i=1}^r Q_{ij}\mathbf 1_i\). Therefore the matrix of \(A|_W\), relative to the basis \(\mathbf 1_1,\ldots,\mathbf 1_r\), is \(Q\). In particular, every eigenvalue of \(Q\) occurs in the adjacency spectrum of \(\Gamma\).

For each \(i\), let
\[
U_i=\left\{v\in\mathbb R^{V(\Gamma)}:\operatorname{supp}(v)\subseteq{\mathcal{C}}_{P_i} \text{ and } \sum_{x\in{\mathcal{C}}_{P_i}}v_x=0 \right\}.
\]
Then \(\dim U_i=c_i-1\), and \(\mathbb R^{V(\Gamma)} =W\oplus U_1\oplus\cdots\oplus U_r\).

Suppose first that \(P_i\neq\varnothing\). Then \({\mathcal{C}}_{P_i}\) is independent. Moreover, every vertex outside \({\mathcal{C}}_{P_i}\) is adjacent either to all vertices of this class or to none of them. If \(v\in U_i\), the contribution at any vertex outside the class is therefore either zero or the sum of the coordinates of \(v\), which is zero. There is no contribution inside the class. Hence \(Av=0\) for every \(v\in U_i\). Thus the nonempty class \({\mathcal{C}}_{P_i}\) contributes \(c_i-1\) copies of the eigenvalue \(0\). Summing over the nonempty represented classes gives exactly the stated value of \(\mu_0\).

Now suppose that \(P_i=\varnothing\). Then \({\mathcal{C}}_{P_i}={\mathcal{C}}_\varnothing\) is a clique and is completely joined to every other represented class. Let \(v\in U_i\). As before, the coordinates of \(Av\) outside \({\mathcal{C}}_\varnothing\) vanish. For \(x\in{\mathcal{C}}_\varnothing\),
\[
(Av)_x = \sum_{\substack{y\in{\mathcal{C}}_\varnothing\\y\neq x}}v_y = -v_x,
\]
because the coordinates of \(v\) sum to zero. Hence \(Av=-v\) on \(U_i\). Thus, the empty-support class contributes \(|{\mathcal{C}}_\varnothing|-1\) copies of the eigenvalue \(-1\).

Finally,
\[
\dim W+\sum_{i=1}^r\dim U_i = r+\sum_{i=1}^r(c_i-1) = \sum_{i=1}^r c_i = |V(\Gamma)|.
\]
The above invariant subspaces, therefore, exhaust the whole adjacency space. This gives the spectrum, and the characteristic polynomial factorization follows at once.
\end{proof}

\begin{corollary}
\label{cor:rank-nullity}
Let \(A\) be the adjacency matrix of \(\Gamma_H(M)\), and let \(m_B(\lambda)\) denote the algebraic multiplicity of \(\lambda\) as an eigenvalue of a square matrix \(B\). Then
\[
\operatorname{nullity}(A)=\mu_0+m_Q(0), \operatorname{rank}(A) = |V(\Gamma_H(M))|-\mu_0-m_Q(0).
\]
Moreover, \(m_A(-1)=\mu_{-1}+m_Q(-1)\). In particular, \(A\) is singular if and only if \(\mu_0>0\) or \(0\in\Spec(Q)\).
\end{corollary}

\vskip 5mm
\noindent{\bf Concluding remarks:}
\vskip 1mm

For the modules considered here, adjacency in \(\Gamma_H(M)\) is determined by the multiplicity-one blocks. The two basic formulas are
\(
(Rx:M)=\bigcap_{\alpha\notin\sigma(x)}\Ann_R(S_\alpha),
x\leftrightarrow y
\text{ if and only if }
\sigma(x)\cap\sigma(y)=\varnothing.
\)
Thus multiplicities greater than one do not enter the adjacency criterion; they affect only the sizes of the obstruction classes. The resulting quotient is a weighted disjoint graph on subsets of \({I_1}\). The reconstruction results also identify the two limitations of this description: universal singleton classes may be indistinguishable from the empty-support class, and in the multiplicity-free case the full face of the incidence complex need not be represented by a vertex.

\vskip 1cm

\noindent{\bf Conflict of Interest}: The authors declare that they have no conflicts of interest.\\
\noindent{\bf Data Availability}: Data sharing is not applicable to this article.

\end{document}